\documentclass[11pt]{article}
\usepackage{amsmath}
\usepackage{amsthm}
\usepackage{amsfonts}
\usepackage{amssymb}
\usepackage{latexsym}
\usepackage{diagbox}
\usepackage{tikz}
\usepackage{tkz-graph}
\usepackage{graphicx}
\usetikzlibrary{matrix,arrows}
\usetikzlibrary{fit}
\usetikzlibrary{patterns}
\usetikzlibrary{chains,fit,shapes}
\usetikzlibrary{positioning,calc}
\usetikzlibrary{graphs}
\usetikzlibrary{graphs.standard}
\usetikzlibrary{arrows,decorations.markings}
\usepackage[colorlinks,
linkcolor=blue,
anchorcolor=blue,
citecolor=blue
]{hyperref}

\newtheorem{thm}{Theorem}
\newtheorem{lem}[thm]{Lemma}

\newtheorem{conj}{Conjecture}
\newtheorem{rem}[thm]{Remark}

\DeclareMathOperator{\asc}{asc}
\DeclareMathOperator{\des}{des}
\DeclareMathOperator{\rmax}{rmax}
\DeclareMathOperator{\lmax}{lmax}
\DeclareMathOperator{\rmin}{rmin}
\DeclareMathOperator{\lmin}{lmin}
\DeclareMathOperator{\stat}{stat}

\title{Distributions of statistics on separable permutations}

\author{Joanna N. Chen\footnote{College of Science, Tianjin University of Technology, Tianjin 300384, P. R. China.  {\bf Email:} joannachen@tjut.edu.cn}\ , Sergey Kitaev\footnote{Department of Mathematics and Statistics, University of Strathclyde, 26 Richmond Street, Glasgow G1 1XH, United Kingdom. 
{\bf Email:} sergey.kitaev@strath.ac.uk.}\ \ and Philip B. Zhang\footnote{College of Mathematical Science, Tianjin Normal University, Tianjin  300387, P. R. China.  {\bf Email:} zhang@tjnu.edu.cn.}}

\begin{document}

\maketitle 

\noindent\textbf{Abstract.} 
We derive functional equations for distributions of six classical statistics (ascents, descents, left-to-right maxima, right-to-left maxima, left-to-right minima, and right-to-left minima) on separable and irreducible separable permutations. The equations are used to find a third degree equation for joint distribution of ascents and descents on separable permutations that generalizes the respective known result for the descent distribution. Moreover, our general functional equations allow us to derive explicitly (joint) distribution of any subset of maxima and minima statistics on irreducible, reducible and all separable permutations. In particular, there are two equivalence classes of distributions of a pair of maxima or minima statistics. Finally, we  present three unimodality conjectures about distributions of statistics on separable permutations.  \\

\noindent {\bf AMS Classification 2010:} 05A15


\section{Introduction}\label{intro-sec}

A permutation of length $n$ is a rearrangement of the set $[n]:=\{1,2,\ldots,n\}$. Denote by  $S_n$  the set of permutations of $[n]$ and let $\varepsilon$ be the empty permutation. A {\em pattern} is a permutation. A permutation $\pi_1\pi_2\cdots\pi_n\in S_n$ avoids a pattern $p=p_1p_2\cdots p_k\in S_k$  if there is no subsequence $\pi_{i_1}\pi_{i_2}\cdots\pi_{i_k}$ such that $\pi_{i_j}<\pi_{i_m}$ if and only if $p_j<p_m$.  Let $S_n(p)$ denote the set of $p$-avoiding permutations of length $n$. For a permutation $\pi=\pi_1\pi_2\cdots\pi_n$, its {\em reverse} is the permutation $r(\pi)=\pi_n\pi_{n-1}\cdots \pi_1$ and its {\em complement} is the permutation $c(\pi)=c(\pi_1)c(\pi_2)\cdots c(\pi_n)$ where $c(x)=n+1-x$. Also, the length of a permutation $\pi$, denoted by $|\pi|$, is the number of elements in $\pi$. For example, for $\pi=423165$, $|\pi|=6$, $r(\pi)=561324$, and $c(\pi)=354612$.

For $n\geq 2$, a permutation $\pi_1\pi_2\cdots\pi_n$ is {\em irreducible} if there is no $i$, $2\leq i\leq n$, such that any element in $\pi_1\cdots\pi_{i-1}$ is less than each element in $\pi_i\cdots\pi_n$. For example, the permutation $42513$ is irreducible. By definition, the permutation 1 is irreducible. A permutation is {\em reducible} if it is not irreducible. The empty permutation $\varepsilon$ is neither reducible nor irreducible. Each reducible permutation can be cut into irreducible pieces called {\em irreducible components}, or just {\em components}. For example, the components of  the permutation 312546978 are 312, 54, 6, and 978.

\subsection{Separable permutations}
Suppose $\pi=\pi_1\pi_2\cdots\pi_m \in S_m$ and $\sigma=\sigma_1\sigma_2\cdots\sigma_n\in S_n$. We define the {\em direct sum}\index{direct sum $\oplus$} (or simply, {\em sum}) $\oplus$, and the {\em skew sum}\index{skew sum $\ominus$} $\ominus$ by building the permutations $\pi\oplus\sigma$ and $\pi\ominus\sigma$ as follows:
\begin{eqnarray*}
(\pi\oplus\sigma)_i&=& \left\{
\begin{array}{ll}
\pi_i  &  \mbox{ if }1\leq i\leq m,\\
\sigma_{i-m}+m & \mbox{ if }m+1\leq i\leq m+n,
\end{array}\right.\nonumber \\
(\pi\ominus\sigma)_i &=& \left\{
\begin{array}{ll}
\pi_i+n  &  \mbox{ if }1\leq i\leq m,\\
\sigma_{i-m} & \mbox{ if }m+1\leq i\leq m+n.
\end{array}\right.\nonumber
\end{eqnarray*}
For example, $14325\oplus 4231=143259786$ and $14325\ominus 4231= 587694231$. 

The {\em separable permutations} are those which can be built from the permutation $1$ by repeatedly applying the $\oplus$ and $\ominus$ operations. Bose, Buss and Lubiw~\cite{BBL98} introduced the notion of separable permutation in 1998, and it is well-known \cite{Kitaev2011Patterns} that the set of all separable permutations of length $n\geq 1$ is precisely $S_n(2413,3142)$ (we assume here that by definition, the empty permutation $\varepsilon$ is not separable).  Separable permutations appear in the literature in various contexts (see, for example, \cite{AAV2011,AHP2015,AJ2016,FuLinZeng,GaoLiu,NRV,Stankova1994}). Cleary, the reverse, or complement, or (usual group theoretic) inverse, or any composition of these operations applied to a separable permutation gives a separable permutation. 

It is not difficult to see that any $\pi\in S_n(2413,3142)$ has the following
structure (see also Figure~\ref{sepStructure} for a schematic representation, where a permutation is viewed as a diagram with a dot in position $(i,\pi_i)$ for each element $\pi_i$ of $\pi$):
\begin{equation}\label{structureSep}\pi=L_1L_2\cdots
L_mnR_mR_{m-1}\cdots R_1\notag \end{equation} where
\begin{itemize}
\item for $1\leq i\leq m$, $L_i$ and $R_i$ are non-empty ($\neq \varepsilon$), with possible exception for
$L_1$ and $R_m$, separable permutations which are intervals in $\pi$ (that is, consist of all elements in $\{a,a+1,\ldots,b\}$ for some $a$ and $b$);
\item $L_1<R_1<L_2<R_2<\cdots <L_m<R_m$, where $A<B$, for two
permutations $A$ and $B$, means that each element of $A$ is less than
every element of $B$. In particular, $L_1$, if it is not empty,
contains 1.
\end{itemize}

\noindent
For example, if $\pi=2165743$ then $L_1=21$, $L_2=65$, $R_1=43$ and
$R_2=\varepsilon$. Note that a permutation in question is irreducible if and only if $L_1=\varepsilon$.

\begin{figure}[ht]
\begin{center}
\begin{tikzpicture}[line width=0.5pt,scale=0.24]
\coordinate (O) at (0,0);

\path (25,1)  node {$n$};
\draw [dashed] (O)--++(50,0);
\fill[black!100] (O)++(25,0) circle(1.5ex);

\draw (26,-1) rectangle (30,-3);
\path (28,-2)  node {$R_m$};
\path (36.5,-2)  node {$\boldsymbol{\leftarrow}$ possibly empty};

\draw [dashed] (0,-4)--++(50,0);
\draw [dashed] (0,-8)--++(50,0);
\draw [dashed] (0,-17)--++(50,0);
\draw [dashed] (0,-21)--++(50,0);
\draw [dashed] (0,-25)--++(50,0);

\draw [dashed] (25,0)--++(0,-29);
\draw [dashed] (31,0)--++(0,-29);
\draw [dashed] (37,0)--++(0,-29);
\draw [dashed] (43,0)--++(0,-29);

\draw [dashed] (19,0)--++(0,-29);
\draw [dashed] (13,0)--++(0,-29);
\draw [dashed] (7,0)--++(0,-29);

\draw (20,-5) rectangle (24,-7);
\path (22,-6)  node {$L_m$};

\fill[black!100] (O)++(32,-9) circle(1.0ex);
\fill[black!100] (O)++(33,-10) circle(1.0ex);
\fill[black!100] (O)++(34,-11) circle(1.0ex);
\draw (38,-14) rectangle (42,-16);
\path (40,-15)  node {$R_2$};

\draw (45,-22) rectangle (49,-24);
\path (47,-23)  node {$R_1$};

\fill[black!100] (O)++(17,-11) circle(1.0ex);
\fill[black!100] (O)++(16,-12) circle(1.0ex);
\fill[black!100] (O)++(15,-13) circle(1.0ex);
\draw (8,-18) rectangle (12,-20);
\path (10,-19)  node {$L_2$};

\draw (2,-26) rectangle (6,-28);
\path (4,-27)  node {$L_1$};
\path (12.5,-27)  node {$\boldsymbol{\leftarrow}$ possibly empty};

\end{tikzpicture}
\caption{A schematic view of the permutation diagrams corresponding to separable permutations. Each $L_i$ and $R_j$ is a separable permutation.}\label{sepStructure}
\end{center}
\end{figure}
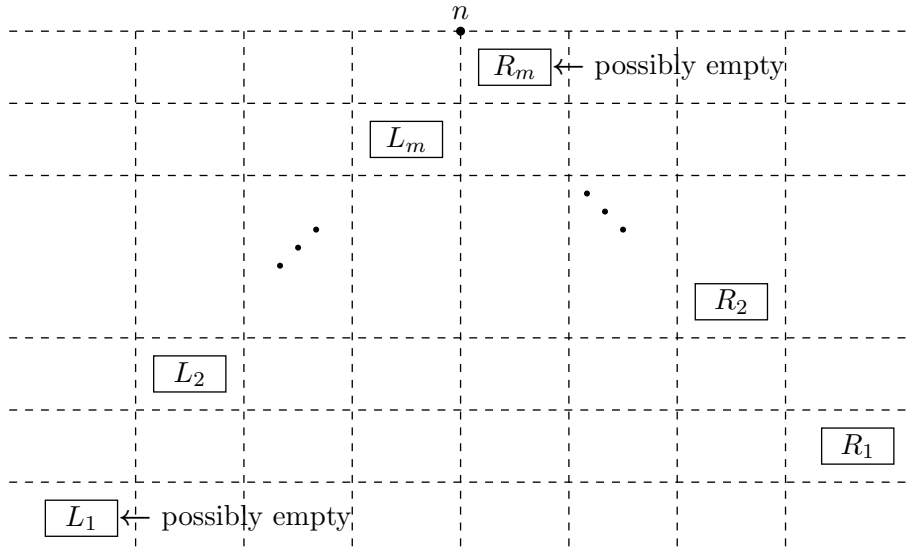

It was conjectured by Shapiro and Getu and, for the first time, proved by West~\cite{West1995}
that $S_n(3142,2413)$ is counted by the $(n-1)$-th {\em Schr\"oder number} (this is the sequence \cite[A006318]{oeis} beginning with 1, 2, 6, 22, 90, 394, 1806,...). The proof involves studying the
generating tree for the restricted permutations and it uses a well-known relation between the Schr\"oder numbers $s_n$ and the Catalan
numbers $C_n$:
\begin{equation}\label{shr-numbers}s_n=\sum_{i=0}^{n}{2n-i \choose
i}C_{n-i}.\end{equation}
Several other proofs of the enumeration formula for $S_n(3142,2413)$ appear in the literature (see \cite[Sec 2.2.5]{Kitaev2011Patterns}), and an alternative to formula (\ref{shr-numbers}) was used in~\cite{Reif2003}:
\begin{equation}\label{shr-new}s_n=\sum_{k=0}^{n}2^kC_{n,k},\end{equation}
where $C_{n,k}$ is the number of Dyck paths of length $2n$ with $k$
peaks. The generating function (g.f.) for separable permutations is
\begin{equation}\label{gf-sep-perms}
S(t):=\sum_{n\geq 1}|S_n(3142,2413)|t^n=\frac{1-t-\sqrt{1-6t+t^2}}{2}.
\end{equation}
In our studies, we also need the g.f. $I(t)$ for irreducible separable permutations. Note that for $n\geq 2$, the number of irreducible separable permutations in $S_n$ is the same as the number of reducible separable permutations, which can be seen by applying the reverse operation to all separable permutations (sending reducible permutations to irreducible, and vice versa). Hence, 
\begin{equation}\label{irr-sep-perms-gf}
I(t)=t+\frac{1}{2}(S(t)-t)=\frac{1+t-\sqrt{1-6t+t^2}}{4}.
\end{equation}
The sequence of irreducible separable permutations begins with 1, 1, 3, 11, 45, 197, 903,... and these are the little Schr\"{o}der numbers \cite[A001003]{oeis}. 

\subsection{Distributions of permutation statistics}
For a permutation $\pi=\pi_1 \pi_2 \cdots \pi_n$,  the {\em descent} (resp., {\em ascent}) {\em statistic} on $\pi$, $\des(\pi)$ (resp., $\asc(\pi)$), is defined as the number of $i\in \{1,2,\ldots,n-1\}$ such that $\pi_i > \pi_{i+1}$ (resp., $\pi_i<\pi_{i+1}$). For example, des$(561423)=2$ and asc$(25134)=3$. The distribution of descents (or ascents) over $S_n$ is given by the  {\em Eulerian polynomial} 
\begin{align*}
A_n(q) := \sum_{\pi \in S_n}q^{\des (\pi)}=\sum_{k=1}^{n}k!S(n,k)(q-1)^{n-k}
\end{align*}
where $S(n,k)$ is the {\em Stirling number of the second kind}. 

The distribution of descents has been studied over various subsets of permutations in the literature. For example, the distribution of descents over $S_n(231)$ is given by the  {\em Narayana numbers}  $$\frac{1}{n}\binom{n}{k}\binom{n}{k+1}$$ (see  \cite{Stanley1989}). Also, the distribution of descents over 123-avoiding permutations is given by the following formula \cite{BBS2010,BKLPRW}, where $t$ and $q$ correspond to the length and the number of descents:
\begin{align}
\frac{-1+2tq+2t^2q-2tq^2-4t^2q^2+2t^2q^3+\sqrt{1-4tq-4t^2q+4t^2q^2}}{2tq^2(tq-1-t)}. \notag
\end{align}
The number of  321-avoiding permutations of length $n$ with $k$ descents \cite[A091156]{oeis} is given by
$$\frac{1}{n+1}{n+1\choose k}\sum_{j=0}^{n-2k}{k+j-1\choose k-1}{n+1-k \choose n-2k-j}.$$ 
More relevant to our paper is the result in \cite[Cor 2.4]{FuLinZeng} stating that the g.f. $S(t,q)$ giving the distribution of descents (equivalently, ascents, because of the reverse operation) on separable permutations satisfies
\begin{equation}\label{des-separable-distr}qS^3(t,q)+qtS^2(t,q)+((1+q)t-1)S(t,q)+t=0.
\end{equation}

For a permutation $\pi=\pi_1\pi_2\cdots \pi_n$, $\pi_i$ is a {\em right-to-left maximum} (resp., {\em right-to-left minimum}) if $\pi_i>\pi_j$ (resp., $\pi_i<\pi_j$) for all $j> i$. In particular, $\pi_n$ is a right-to-left maximum and a right-to-left minimum. The number of right-to-left maxima (resp., right-to-left minima) is denoted by  $\rmax(\pi)$ (resp.,  $\rmin(\pi)$). For example, $\rmax(426513)=3$ and $\rmin(426153)=2$.  Also,  $\pi_i$ is a {\em left-to-right maximum} (resp., {\em left-to-right minimum}), denoted by  $\lmax(\pi)$ (resp.,  $\lmin(\pi)$)  if $\pi_i>\pi_j$ (resp., $\pi_i<\pi_j$) for all $j< i$. In particular, $\pi_1$ is a left-to-right maximum and a left-to-right minimum. For example, $\lmax(425163)=\lmin(426153)=3$.

Note that the distribution of $\rmax$, which is the same as the distribution of $\rmin$, or $\lmax$, or $\lmin$, or cycles in (unrestricted) permutations, is given by the {\em signless Stirling numbers of the first kind}. It is easy to see, and is well-known that
$$\sum_{\pi\in S_n}y^{\rmax(\pi)}=y(y+1)\cdots (y+n-1)=y^{(n)}$$
is given by the {\em rising factorial}.

Two $k$-tuples of (permutation) statistics $(s_1, s_2, \ldots , s_k)$ and $(s_1^\prime, s_2^\prime, \ldots , s_k^\prime)$ are
{\em equidistributed} over a set $S$ if
$$\sum_{a\in S}t^{|a|}t_1^{s_1(a)}t_2^{s_2(a)}\cdots t_k^{s_k(a)}=
\sum_{a\in S}t^{|a|}t_1^{s_1^\prime(a)}t_2^{s_2^\prime(a)}\cdots t_k^{s_k^\prime(a)}$$
where $|a|$ is the size of $a$ (for permutations, the size is the length).

%

In this paper, we are interested in the following g.f.
$$S(t,p,q,x,y,u,v):=\sum_{\pi}t^{|\pi|}p^{\asc(\pi)}q^{\des(\pi)}x^{\lmax(\pi)}y^{\rmax(\pi)}u^{\lmin(\pi)}v^{\rmin(\pi)}$$
where the sum is taken over all separable permutations, and $I(t,p,q,x,y,u,v)$ is the respective g.f. on irreducible separable permutations. For brevity, throughout the paper we often omit the variables in the arguments that are set to be 1. For example, $S(t,p,q,1,y,1,1)$ and $I(t,p,q,1,1,1,1)$ can be denoted by us by $S(t,p,q,y)$ and $I(t,p,q)$, respectively. This will not cause any confusion.

There are 6 unordered pairs and 4 unordered triples of statistics in $\{\lmax,\rmax, \lmin, \rmin\}$. A permutation of statistics in a $k$-tuple of statistics results in a $k$-tuple of statistics having the same distribution modulo renaming variables, and we need to consider only one pair or triple from each such equivalence class. The following lemma shows that with respect to distributions,  there is one equivalence class for a single statistic, one equivalence class for triples of statistics, and two equivalence classes for pairs of statistics.  

\begin{lem}\label{equid-lemma} The (pairs, tripples of) statistics in each of the sets 
\begin{eqnarray} 
&& \{\lmax,\rmax, \lmin, \rmin\} \label{base-set} \\
&& \{(\lmax,\rmax),(\lmin,\rmin),(\lmin,\lmax),(\rmin,\rmax)\}  \label{set2} \\
&& \{(\rmax, \lmin),(\lmax,\rmin)\} \label{set1}  \\
&&  \{(\lmax,\rmax,\lmin),(\lmin,\rmin,\lmax),(\rmin,\rmax,\lmin),\label{set-tripples} \\
&&(\rmax,\rmin,\lmax)\} \notag \end{eqnarray} are equidistributed on (separable) permutations. \end{lem}

\begin{proof} Applying the reverse and complement operations and their composition to all (separable) permutations, we obtain equidistribution of the single and triple statistics, as well as the pairs in set \eqref{set1}. On the other hand, it is not difficult to see from the definitions that under the inverse operation, the statistic $\rmax$ goes to the statistic $\rmax$, while the statistic $\lmax$ goes to the statistic $\rmin$ proving equidistribution of $(\lmax,\rmax)$ and $(\rmin,\rmax)$. The remaining equidistributions in set \eqref{set2} are obtained by keeping track of transformations of the statistics in $(\lmax,\rmax)$ and $(\rmin,\rmax)$ under the reverse and complement operations.  \end{proof}

\begin{rem}\label{reduction-cases-rem}  By Lemma~\ref{equid-lemma}, to determine (joint) distribution of at most three statistics in the set $\{\lmax,\rmax, \lmin, \rmin\}$, it is sufficient to find the distribution of $\rmax$ and the joint distributions of the pairs of statistics $(\lmax,\rmax)$ and $(\rmax, \lmin)$ and the triple of statistics $(\lmax,\rmax,\lmin)$.
Also, by the definitions, if $\pi$ is a permutation of length at least $2$ then $\pi$ is irreducible if and only if 
\begin{itemize}
\item inverse of $\pi$ is irreducible;
\item reverse of $\pi$ is reducible;
\item complement of $\pi$ is reducible.
\end{itemize} 
\end{rem}

\subsection{Our results in this paper} Firstly, in Section~\ref{joint-all-sec} we will prove the following theorem.

\begin{thm}\label{dist-most-general-thm} The g.f. $S(t,p,q,x,y,u,v)$  and $I(t,p,q,x,y,u,v)$ satisfy functional equations (\ref{gf-all-stat-eqn}) and (\ref{gf-irr-stat-eqn}).
\end{thm}

As an immediate corollary of Theorem~\ref{dist-most-general-thm}, by setting all variables but $t$ to be 1, we obtain the g.f.'s \eqref{gf-sep-perms} and~\eqref{irr-sep-perms-gf}.  As another corollary of Theorem~\ref{dist-most-general-thm} (the case of $x=y=u=v=1$) we will obtain the following theorem generalizing the known relation (\ref{des-separable-distr}) for the distribution of descents on separable permutations (substituting $p=1$ in (\ref{asc-des-relation}) gives (\ref{des-separable-distr})).

\begin{thm}\label{dist-asc-des-thm} The g.f. $S(t,p,q)$ and $I(t,p,q)$ giving the joint distribution of ascents and descents on (irreducible) separable permutations satisfy
\begin{equation}\label{asc-des-relation} p q (S(t,p,q))^3 +  p q (S(t,p,q))^2 t + S(t,p,q) ((p  + q) t - 1) + t = 0,\end{equation}
\begin{equation}\label{asc-des-irr-relation} I(t,p,q)=\frac{t+q(t+S(t,p,q))S(t,p,q)}{1+qS(t,p,q)}.\end{equation}
\end{thm}

Furthermore, letting all variables in Theorem~\ref{dist-most-general-thm}, but $t$ and $y$, be equal to~1, then replacing $y$ by $z$ and applying Remark~\ref{reduction-cases-rem}, we will justify the following theorem. 

\begin{thm}\label{dist-rmax-thm} The g.f. $S(t,z):=\sum_{\pi}t^{|\pi|}z^{\stat(\pi)}$ giving the distribution of any statistic $\stat\in \{\lmax,\rmax,\lmin,\rmin\}$ on separable permutations is 
\begin{footnotesize}
$$\frac{4\sqrt{\left(-\frac{1}{4} \sqrt{t^2-6 t+1}-tz+\frac{t}{4}+\frac{5}{4}\right)^2+\sqrt{t^2-6 t+1}-t-1}-\sqrt{t^2-6 t+1}+4t z+t - 3}{2 \left(\sqrt{t^2-6 t+1}-t-1\right)}.$$
\end{footnotesize}
The distributions of $\stat$ on irreducible and reducible separable permutations are given in Table~\ref{tab-dist-rmax-thm}. (Tables~\ref{tab-rmax-sep}, \ref{tab-rmax-irr-sep}, \ref{tab-rmin-irr-sep} give initial values for the distributions.)
\end{thm}

\begin{table}
\begin{center}
\begin{tabular}{c|c|c}
g.f. & irreducible & reducible \\
\hline
& & \\[-3mm]
$\frac{S^2(t,z)}{1+S(t,z)}+zt$ & $\stat\in\{\rmax,\lmin\}$ & $\stat\in\{\lmax,\rmin\}$ \\[3mm]
\hline
& & \\[-3mm]
$\frac{S(t,z)}{1+S(t,z)}-zt$ & $\stat\in\{\lmax,\rmin\}$ & $\stat\in\{\rmax,\lmin\}$ 
\end{tabular}
\caption{Distributions of statistics in Theorem~\ref{dist-rmax-thm}}\label{tab-dist-rmax-thm}
\end{center}
\end{table}

As for joint equidistributions, we will prove the following four theorems.

%
\begin{thm}\label{dist-lmax-rmax-thm}  The g.f. $S(t,z_1,z_2):=\sum_{\pi}t^{|\pi|}z_1^{\stat_1(\pi)}z_2^{\stat_2(\pi)}$ giving the distribution of $(\stat_1,\stat_2)$ in set \eqref{set2} on separable permutations is 
$$\frac{(S(t,z_1)+1) (S(t,z_2)+1) t z_1 z_2}{1-S(t,z_1) S(t,z_2)}$$
where $S(t,z)$ is given in Theorem~\ref{dist-rmax-thm}. The distributions of  $$(\stat_1,\stat_2)\in\{(\lmax,\rmax),(\lmax,\lmin),(\rmin,\rmax),(\rmin,\lmin)\}$$ on irreducible and reducible separable permutations are given by $\frac{(S(t,z_2)+1) t z_1 z_2}{1-S(t,z_1) S(t,z_2)}$  and $\frac{S(t,z_1)(S(t,z_2)+1) t z_1 z_2}{1-S(t,z_1) S(t,z_2)}$, respectively.
\end{thm}

\begin{thm}\label{dist-rmax-lmin-thm}  The g.f. $S(t,z_1,z_2):=\sum_{\pi}t^{|\pi|}z_1^{\stat_1(\pi)}z_2^{\stat_2(\pi)}$ giving the distribution of $(\stat_1,\stat_2)$ in set \eqref{set1} on separable permutations is  $$\frac{S(t,z_1)I(t,z_2) + t z_1z_2}{1 - S(t,z_1)I(t,z_2) - t z_1 z_2}$$
where $S(t,z_1)$ and $I(t,z_2)$ are given in Theorem~\ref{dist-rmax-thm}. The distributions of  $(\stat_1,\stat_2)$ on irreducible and reducible separable permutations are recorded in Table~\ref{tab-dist-rmax-lmin-thm}.
\end{thm}

In Theorems~\ref{dist-triple-thm} and~\ref{dist-general-xyuv-thm} we use the following function:
\begin{small}
\begin{equation}\label{func-E}E(t,z_1,z_2,z_3):=z_1z_2z_3t+\frac{(S(t,z_1)+1)(S(t,z_2)+1)(S(t,z_3)+1)t^2z_1^2z_2z_3}{(1-S(t,z_1)S(t,z_3))(1-S(t,z_1)S(t,z_2))}.\end{equation}
\end{small}

\begin{thm}\label{dist-triple-thm}  The g.f. $S(t,z_1,z_2,z_3):=\sum_{\pi}t^{|\pi|}z_1^{\stat_1(\pi)}z_2^{\stat_2(\pi)}z_3^{\stat_3(\pi)}$ giving the distribution of $(\stat_1,\stat_2,\stat_3)$ in set \eqref{set-tripples} on separable permutations is 
$$\frac{E(t,z_1,z_2,z_3)}{1-A(z_2,z_3)-tz_2z_3}$$
where $A(z_2,z_3)=S(t,z_2)\left(z_3t+\frac{S^2(t,z_3)}{S(t,z_3)+1}\right)$ and $S(t,z)$ is given in Theorem~\ref{dist-rmax-thm}. The distributions of  $(stat_1,\stat_2,\stat_3)$ in set \eqref{set-tripples} on irreducible separable permutations (denoted by $I(t,z_1,z_2,z_3)$) and reducible separable permutations is given by the following two cases:
\begin{itemize}
\item The distribution of  $(stat_1,\stat_2,\stat_3)$ in $$\{(\lmax,\rmax,\lmin),(\rmin,\rmax,\lmin)\}$$ on reducible separable permutations is given by 
\begin{equation}\label{reducible-three-stat}E(t,z_1,z_2,z_3)-z_1z_2z_3t\end{equation}
and on irreducible separable permutations by 
\begin{small}
$$\frac{1}{1-A(z_2,z_3)-tz_2z_3}\left(xyut+(A(z_2,z_3)+tz_2z_3)(E(t,z_1,z_2,z_3)-z_1z_2z_3t)\right).$$
\end{small}
\item The distribution of  $(stat_1,\stat_2,\stat_3)$ in $$\{(\lmin,\rmin,\lmax),(\rmax,\rmin,\lmax)\}$$ on irreducible separable permutations is given by $E(t,z_1,z_2,z_3)$ and on reducible separable permutations by 
\begin{small}
$$-z_1z_2z_3t+\frac{1}{1-A(z_2,z_3)-tz_2z_3}(xyut+(A(z_2,z_3)+tz_2z_3)(E(t,z_1,z_2,z_3)-z_1z_2z_3t)).$$
\end{small}
\end{itemize}
\end{thm}

\begin{thm}\label{dist-general-xyuv-thm} The g.f. $S(t,x,y,u,v):=\sum_{\pi}t^{|\pi|}x^{\lmax(\pi)}y^{\rmax(\pi)}u^{\lmin(\pi)}v^{\rmin(\pi)}$  giving the distribution of the four statistics over separable permutations is  
$$xyuvt+E(t,x,y,u)S(t,y,u,v)+E(t,x,y,v)S(t,x,u,v)$$
where $S(t,y,u,v)$ and $S(t,x,u,v)$ are given by Theorem~\ref{dist-triple-thm} and $E(t,z_1,z_2,z_3)$ by \eqref{func-E}. The respective g.f. $I(t,x,y,u,v)$ for irreducible separable permutations is  $$xyuvt+E(t,x,y,u)S(t,y,u,v)$$ and for reducible separable permutations is $E(t,x,y,v)S(t,x,u,v)$.
\end{thm}

Finally, in Section~\ref{final-sec} we give several unimodality conjectures confirming which would extend our knowledge in \cite{FuLinZeng} on unimodality of the distribution of descents on separable permutations.  

\begin{table}
\begin{center}
\begin{tabular}{c|c|c}
g.f. & irreducible & reducible \\
\hline
&& \\[-3mm]
$\frac{S(t,z_1)I(t,z_2) + t z_1 z_2}{1 - S(t,z_1)I(t,z_2) - t z_1 z_2} -S(t,z_1)I(t,z_2)$ & $(\rmax,\lmin)$ &  $(\lmax,\rmin)$  \\[2mm]
\hline
&& \\[-3mm]
$S(t,z_1)I(t,z_2)$ & $(\lmax,\rmin)$     &  $(\rmax,\lmin)$  
\end{tabular}
\caption{Distributions of statistics in Theorem~\ref{dist-rmax-lmin-thm} depending on $(\stat_1,\stat_2)$}\label{tab-dist-rmax-lmin-thm}
\end{center}
\end{table}

\section{The key functional equations and specializations}\label{joint-all-sec}

We begin with deriving the following functional equations:
\begin{small}
\begin{eqnarray}\label{gf-all-stat-eqn} & & S(t,p,q,x,y,u,v) = xyuvt + pS(t,p,q,x,1,u,v)I(t,p,q,x,y,1,v)+  \\ \nonumber
& &  \hspace{0.5cm}q(S(t,p,q,x,y,u,1)-I(t,p,q,x,y,u,1)+xyut)S(t,p,q,1,y,u,v);\end{eqnarray}

\vspace{-1cm}

\begin{eqnarray}\label{gf-irr-stat-eqn} & & I(t,p,q,x,y,u,v) =  xyuvt +  \\ \nonumber
& &     \hspace{0.5cm}q(S(t,p,q,x,y,u,1)-I(t,p,q,x,y,u,1)+xyut)S(t,p,q,1,y,u,v). \nonumber
\end{eqnarray}
\end{small}

Referring to the structure of a separable permutation $\pi$ in Figure~\ref{sepStructure}, we have the following four disjoint possibilities, where Cases 1--3 cover irreducible permutations and Case 4 covers reducible permutations. 

\begin{enumerate}
\item $\pi=1$ giving the term of $xyuvt$.
\item $\pi=nR_1=1\ominus R_1$, $R_1\neq\varepsilon$, giving the term of $qxyutS(t,p,q,1,y,u,v)$ because $\lmax(R_1)$ does not contribute to $\lmax(\pi)=1$ and $R_1$ can be any separable permutation, while $n$ contributes an extra descent and one extra left-to-right minimum and one extra right-to-left maximum.
\item $\pi=\pi'\ominus R_1$ where $R_1\neq\varepsilon$ ($L_1=\varepsilon$) and $\pi'$ is a reducible separable permutation (of length at least 2). In this case, we have the term of 
$$q(S(t,p,q,x,y,u,1)-I(t,p,q,x,y,u,1))S(t,p,q,1,y,u,v)$$ since 
\begin{itemize}
\item $\pi'$  does not contribute to the statistic $\rmin$ in $\pi$ and hence has g.f. $S(t,p,q,x,y,u,1)-I(t,p,q,x,y,u,1)$,
\item $R_1$ does not contribute to the statistic $\lmax$ in $\pi$ and hence has g.f. $S(t,p,q,1,y,u,v)$, 
\item $\pi'$  is independent from $R_1$, and
\item an extra descent is formed between $\pi'$ and $R_1$.
\end{itemize}
\item $\pi=L_1\oplus \pi''$ where $L_1\neq\varepsilon$ and $\pi''$ is an irreducible separable permutation. In this case, we have the term of 
$pS(t,p,q,x,1,u,v)I(t,p,q,x,y,1,v)$ since 
\begin{itemize}
\item $L_1$ does not contribute to the statistic $\rmax$ in $\pi$,
\item $\pi''$ does not contribute to the statistic $\lmin$ in $\pi$,
\item $L_1$ is independent from $\pi''$, and 
\item an extra ascent is formed between $L_1$ and $\pi''$.
\end{itemize}
\end{enumerate}

\noindent
Summarising Cases 1--4 gives \eqref{gf-all-stat-eqn}. Summarising Cases 1--3 gives \eqref{gf-irr-stat-eqn}.   Hence, Theorem~\ref{dist-most-general-thm} holds.

The following theorem will be useful for us to rewrite, in several occasions, special cases coming from  \eqref{gf-all-stat-eqn} (after setting some of the variables be equal to 1) in a shorter form.

\begin{thm}\label{useful-relation} The following equalities hold
\begin{small}
\begin{eqnarray}
 S(t,p,q,x,y,u,v)-I(t,p,q,x,y,u,v) \ = &pI(t,p,q,x,u,v)S(t,p,q,x,y,v) \notag \\ 
  = &pI(t,p,q,x,y,v)S(t,p,q,x,u,v). \label{useful-formula} 
\end{eqnarray}
\end{small}\end{thm} 

\begin{proof} Clearly, $S(t,p,q,x,y,u,v)-I(t,p,q,x,y,u,v)$ is the g.f. for all reducible separable permutations. We have two more ways to generate all such permutations. Indeed, each reducible permutation $\pi$ is of the form $\pi'\oplus\pi''$, and we can think of $\pi'$ and $\pi''$ in two ways obtaining the equalities:
\begin{itemize}
\item $\pi'$ is an irreducible separable permutation and $\pi''$ is any separable permutation; 
\item $\pi'$ is any separable permutation and $\pi''$ is an irreducible separable permutation. 
\end{itemize}
In both cases $\rmax(\pi')$ (resp., $\lmin(\pi'')$) does not contribute to $\rmax(\pi)$ (resp., $\lmin(\pi)$) and hence the variable $y$ (resp., $u$) must be set to 1 in the respective g.f.'s; also, the factor of $q$ corresponds to the extra ascent formed between $\pi'$ and $\pi''$.  
\end{proof}

The following theorem is a direct consequence to Remark~\ref{reduction-cases-rem} and the fact that the permutation 1 is irreducible, and it is used by us implicitly when considering a single set of statistics from an equivalence class in proofs of theorems.

\begin{thm}\label{gf-transformation} Suppose  that $I$ and $R$ are g.f.'s giving joint distributions of statistics in $(\stat_1,\ldots,\stat_m)$ recorded, respectively, by variables in $\{z_1,\ldots,z_m\}$ on irreducible and reducible separable permutations, respectively (and $t$ records the length of permutations). Further, let statistic $\stat_i$ goes to statistic $\stat'_i$ for $1\leq i\leq m$ under the reverse (or complement) operation. Then $I-z_1\cdots z_mt$ and $R+z_1\cdots z_mt$ are g.f. for joint distributions of  statistics in $(\stat'_1,\ldots,\stat'_m)$ on reducible and irreducible separable permutations, respectively. \end{thm}

\subsection{Joint distribution of ascents and descents}\label{dist-asc-des-sec}

Our next goal is proving Theorem~\ref{dist-asc-des-thm} for
$$S(t,p,q)=\sum_{\pi}t^{|\pi|}p^{\asc(\pi)}q^{\des(\pi)}$$
where the sum is taken over all separable permutations, and for $I(t,p,q)$ obtained by restricting $S(t,p,q)$ to irreducible permutations. 

By (\ref{useful-formula}), we have
\begin{equation}\label{pIS=S-I} S(t,p,q)-I(t,p,q)=pI(t,p,q)S(t,p,q). \end{equation}
Setting $x=y=u=v=1$ in (\ref{gf-all-stat-eqn}) and (\ref{gf-irr-stat-eqn}), and using (\ref{pIS=S-I}) in the specialization of (\ref{gf-all-stat-eqn}), we obtain
\begin{equation}\label{gf-asc-des} S(t,p,q)=t(qS(t,p,q)+1)+pqIS^2(t,p,q)+pI(t,p,q)S(t,p,q);\end{equation}
\begin{equation}\label{irr-gf-asc-des} I(t,p,q)= t(qS(t,p,q)+1)+ q(S(t,p,q)-I)S(t,p,q).\end{equation}
Solving  (\ref{irr-gf-asc-des}) for $I(t,p,q)$ we obtain \eqref{asc-des-irr-relation}, and substituting the solution into (\ref{gf-asc-des}), we obtain \eqref{asc-des-relation} hence completing our proof of Theorem~\ref{dist-asc-des-thm}.

\subsection{Distribution of a statistic in $\{\lmax,\rmax, \lmin, \rmin\}$}\label{dist-rmax-sec}

By Remark~\ref{reduction-cases-rem}, to prove Theorem~\ref{dist-rmax-thm} it is sufficient to prove it for the g.f.
$$S(t,y)=\sum_{\pi}t^{|\pi|}y^{\rmax(\pi)}$$
where the sum is taken over all separable permutations, and for the respective g.f. $I(t,y)$ obtained by restricting $S(t,y)$ to irreducible permutations. 

By (\ref{useful-formula}), we have
\begin{equation}\label{I(t)S=S-I} S(t,y)-I(t,y)=I(t)S(t,y)=I(t,y)S(t). \end{equation}
Setting $p=q=x=u=v=1$ in (\ref{gf-all-stat-eqn}) and (\ref{gf-irr-stat-eqn}), and using (\ref{I(t)S=S-I}) in the specialization of (\ref{gf-all-stat-eqn}), we obtain
\begin{equation}\label{gf-rmax} S(t,y)=yt(S(t,y)+1) + I(t)(S(t,y)+1)S(t,y);\end{equation}
\begin{equation}\label{irr-gf-rmax} I(t,y)= yt(S(t,y)+1)+ (S(t,y)-I(t,y))S(t,y).\end{equation}
Solving simultaneously the system of equations (\ref{gf-rmax}) and (\ref{irr-gf-rmax}), and using the observations in Remark~\ref{reduction-cases-rem} on reducibility of permutations under reverse and complement, we complete the proof of Theorem~\ref{dist-rmax-thm}.

In Tables~\ref{tab-rmax-sep},~\ref{tab-rmax-irr-sep} and~\ref{tab-rmin-irr-sep} we present, respectively, initial values for the distribution of a statistic in $\{\lmax,\rmax,\lmin,\rmin\}$ on separable permutations, a statistic in $\{\rmax,\lmin\}$ on irreducible separable permutations, and a statistic in $\{\lmax,\rmin\}$ on irreducible separable permutations.

\subsection{Joint distribution of $\lmax$ and $\rmax$}\label{joint-rmax-lmax-sec}

By Remark~\ref{reduction-cases-rem}, to prove Theorem~\ref{dist-lmax-rmax-thm} it is sufficient to prove it for the g.f.
$$S(t,x,y)=\sum_{\pi}t^{|\pi|}x^{\lmax(\pi)}y^{\rmax(\pi)}$$
where the sum is taken over all separable permutations, and for the respective g.f. $I(t,x,y)$ for irreducible separable permutations.  

By (\ref{useful-formula}), we have
\begin{equation}\label{twice-IS=S-I} S(t,x,y)-I(t,x,y)=S(t,x)I(t,x,y)=S(t,x,y)I(t,x). \end{equation}
Setting $p=q=u=v=1$ in (\ref{gf-all-stat-eqn}) and (\ref{gf-irr-stat-eqn}), and using (\ref{twice-IS=S-I}) in the specialization of (\ref{gf-all-stat-eqn}), we obtain
\begin{equation}\label{gf-rmax-lmax} S(t,x,y)=xyt(S(t,y)+1) + S(t,x,y)I(t,x)(S(t,y)+1);\end{equation}
\begin{equation}\label{irr-gf-rmax-lmax} I(t,x,y) = xyt(S(t,y)+1)+ (S(t,x,y)-I(t,x,y))S(t,y).\end{equation}
Solving simultaneously the system of equations (\ref{gf-rmax-lmax}) and (\ref{irr-gf-rmax-lmax}), and using the observations in Remark~\ref{reduction-cases-rem} on reducibility of permutations under reverse, complement and inverse, and Theorem~\ref{gf-transformation}, we complete the proof of Theorem~\ref{dist-lmax-rmax-thm}.

The initial terms of $S(t,x,y)$ are
\begin{footnotesize}
\begin{eqnarray}
&& t ( xy ) + t^2 ( x^2y + xy^2 ) + t^3 ( x^3y + 2x^2y^2 + xy^3 + x^2y + xy^2 ) + \notag \\
&& t^4 ( x^4y + 3x^3y^2 + 3x^2y^3 + xy^4 + 3x^3y + 4x^2y^2 + 3xy^3 + 2x^2y + 2xy^2 ) + \cdots. \notag
\end{eqnarray}
\end{footnotesize}
The initial terms of $I(t,x,y)$ are
\begin{footnotesize}
\begin{eqnarray}
&& t ( xy ) + t^2 ( xy^2 ) +  t^3 ( x^2y^2 + xy^3 + xy^2 ) + t^4 ( x^3y^2 + 2x^2y^3 + xy^4 + 2x^2y^2 + 3xy^3 + 2xy^2 )+\notag \\
&& t^5 ( x^4y^2 + 3x^3y^3 + 3x^2y^4 + xy^5 + 4x^3y^2 + 7x^2y^3 + 6xy^4 + 5x^2y^2 + 9xy^3 + 6xy^2 ) + \cdots. \notag 
\end{eqnarray}
\end{footnotesize}

\subsection{Joint distribution of $\rmax$ and $\lmin$}\label{joint-rmax-lmin-sec}  By Remark~\ref{reduction-cases-rem}, to prove Theorem~\ref{dist-rmax-lmin-thm} it is sufficient to prove it for the g.f.
$$S(t,y,u)=\sum_{\pi}t^{|\pi|}y^{\rmax(\pi)}u^{\lmin(\pi)}$$
where the sum is taken over all separable permutations, and for the respective g.f. $I(t,y,u)$ for irreducible separable permutations.  

Setting $p=q=x=v=1$ in (\ref{gf-all-stat-eqn}) and (\ref{gf-irr-stat-eqn}) we obtain
\begin{equation}\label{gf-rmax-lmin} S(t,y,u)=yut + S(t,u)I(t,y) + (S(t,y,u)-I(t,y,u)+yut)S(t,y,u);\end{equation}
\begin{equation}\label{irr-gf-rmax-lmin} I(t,y,u) = yut + (S(t,y,u)-I(t,y,u)+yut)S(t,y,u).\end{equation}
Solving simultaneously the system of equations (\ref{gf-rmax-lmin}) and (\ref{irr-gf-rmax-lmin}), and using the observations in Remark~\ref{reduction-cases-rem} on reducibility of permutations under reverse and complement, and Theorem~\ref{gf-transformation}, we complete the proof of Theorem~\ref{dist-rmax-lmin-thm}. Note that in Theorem~\ref{dist-rmax-lmin-thm}, by (\ref{useful-formula}), $S(t,u)I(t,y)=I(t,u)S(t,y)$. 

The initial terms of $S(t,y,u)$ are
\begin{footnotesize}
\begin{eqnarray}
&& t ( uy ) + t^2 ( u^2y^2 + uy ) + t^3 ( u^3y^3 + 2u^2y^2 + u^2y + uy^2 + uy ) + \notag \\
&& t^4 ( u^4y^4 + 3u^3y^3 + 2u^3y^2 + 2u^2y^3 + u^3y + 4u^2y^2 + uy^3 + 3u^2y + 3uy^2 + 2uy )+ \cdots. \notag
\end{eqnarray}
\end{footnotesize}
The initial terms of $I(t,y,u)$ are
\begin{footnotesize}
\begin{eqnarray}
&& t ( uy ) + t^2 ( u^2y^2 ) + t^3 ( u^3y^3 + 2u^2y^2 ) + t^4 ( u^4y^4 + 3u^3y^3 + 2u^3y^2 + 2u^2y^3 + 3u^2y^2 ) + \notag \\
&& t^5 ( u^5y^5 + 4u^4y^4 + 3u^4y^3 + 3u^3y^4 + 2u^4y^2 + 8u^3y^3 + 2u^2y^4 + 8u^3y^2 + 8u^2y^3 + 6u^2y^2 )+\cdots. \notag 
\end{eqnarray}
\end{footnotesize}

\subsection{Joint distribution of $\lmax$, $\rmax$ and $\lmin$}
 By Remark~\ref{reduction-cases-rem}, to prove Theorem~\ref{dist-triple-thm} it is sufficient to prove it for the g.f.
$$S(t,x,y,u)=\sum_{\pi}t^{|\pi|}x^{\lmax(\pi)}y^{\rmax(\pi)}u^{\lmin(\pi)}$$
where the sum is taken over all separable permutations, and for the respective g.f. $I(t,x,y,u)$ for irreducible separable permutations.  

Setting $p=q=v=1$ in (\ref{gf-all-stat-eqn}) and (\ref{gf-irr-stat-eqn}) we obtain
\begin{eqnarray}\label{gf-lmax-rmax-lmin}  S(t,x,y,u) &=& xyut +  S(t,x,u)I(t,x,y) +\\ 
& & (S(t,x,y,u)-I(t,x,y,u)+xyut)S(t,y,u); \notag\end{eqnarray}
\begin{equation}\label{irr-gf-lmax-rmax-lmin} I(t,x,y,u) = xyut + (S(t,x,y,u)-I(t,x,y,u)+xyut)S(t,y,u).\end{equation}

The system of equations (\ref{gf-lmax-rmax-lmin}) and (\ref{irr-gf-lmax-rmax-lmin}) is equivalent to the system
\begin{eqnarray}\label{gf-lmax-rmax-lmin-1} && S(t,x,y,u)= S(t,x,u)I(t,x,y) + I(t,x,y,u) ; \end{eqnarray}
\begin{equation}\label{irr-gf-lmax-rmax-lmin-1} I(t,x,y,u) =\frac{xyut+(S(t,x,y,u)+xyut)S(t,y,u)}{S(t,y,u)+1}.\end{equation}

\begin{table}
\begin{center}
\begin{tabular}{|c|c|c|c|c|c|c|c|c|}
\hline
\diagbox{$n$}{$k$}&  1  & 2  & 3 &  4 & 5 & 6 & 7 & 8\\
\hline
1&  1 & & & & & & & \\
\hline
2&  1 & 1 & & & & & & \\
\hline
3& 2 & 3 & 1& & & & & \\
\hline
4 & 6  & 9 & 6 & 1 & & & & \\
\hline
5&  22 & 31  & 26 & 10 & 1 & &  &\\
\hline
6& 90 & 120 & 108 & 60 & 15 & 1 & & \\
\hline
7&  394 & 504 & 461  & 305 & 120 & 21 & 1 & \\
\hline
8 & 1806  & 2240 & 2046 & 1475 & 745 & 217 & 28 & 1\\
\hline
\end{tabular}
  \caption{Distribution of $\stat\in \{\lmax,\rmax, \lmin, \rmin\}$, where $n$ is the length of permutations and  $k$ is the number of occurrences of $\stat$.}
\label{tab-rmax-sep}
\end{center}
\end{table}

Observe that from \eqref{gf-lmax-rmax-lmin-1}, the distribution of the statistics over reducible permutations is $S(t,x,u)I(t,x,y)$. Solving simultaneously the system of equations (\ref{gf-lmax-rmax-lmin-1}) and (\ref{irr-gf-lmax-rmax-lmin-1}), which involves using the formulas $S(t,x,u)=\frac{(S(t,x)+1)(S(t,u)+1)txu}{1-S(t,x)S(t,u)}$ and $I(t,x,y)=\frac{(S(t,y)+1)txy}{1-S(t,x)S(t,y)}$ from Theorem~\ref{dist-lmax-rmax-thm} and the formula $S(t,y,u)=\frac{S(t,y)I(t,u)+tyu}{1-S(t,y)I(t,u)-tyu}$ from Theorem~\ref{dist-rmax-lmin-thm}, we complete the proof of Theorem~\ref{dist-triple-thm}. 

The initial terms of $S(t,x,y,u)$ are
\begin{footnotesize}
\begin{eqnarray}
&& t ( uxy ) +  t^2 ( u^2xy^2 + ux^2y ) + \notag \\
&& t^3 ( u^3xy^3 + u^2x^2y^2 + u^2x^2y + ux^3y + u^2xy^2 + ux^2y^2 ) + \notag \\
&& t^4 ( u^4xy^4 + u^3x^2y^3 + u^3x^2y^2 + u^2x^3y^2 + 2u^3xy^3 + u^2x^2y^3 + u^3x^2y + 2u^2x^3y + ux^4y +\notag \\
&&  u^3xy^2 + 2u^2x^2y^2 + 2ux^3y^2 + u^2xy^3 + ux^2y^3 + u^2x^2y + ux^3y + u^2xy^2 + ux^2y^2 ) + \cdots. \notag
\end{eqnarray}
\end{footnotesize}
The initial terms of $I(t,x,y,u)$ are
\begin{footnotesize}
\begin{eqnarray}
&& t ( uxy ) + t^2 ( u^2xy^2 ) + t^3 ( u^3xy^3 + u^2x^2y^2 + u^2xy^2 ) +  t^4 ( u^4xy^4 + u^3x^2y^3 + \notag \\
&& u^3x^2y^2 +u^2x^3y^2 + 2u^3xy^3 + u^2x^2y^3 + u^3xy^2 + u^2x^2y^2 + u^2xy^3 + u^2xy^2 ) + \cdots. \notag
\end{eqnarray}
\end{footnotesize}

\subsection{Joint distribution of $\lmax$, $\rmax$, $\lmin$ and $\rmin$}

Setting $p=q=1$ in \eqref{gf-all-stat-eqn} and \eqref{gf-irr-stat-eqn}, we obtain
\begin{eqnarray}\label{gf-all-stat-eqn-but-two} S(t,x,y,u,v) &=&  xyuvt +  S(t,x,u,v)I(t,x,y,v)+ \\ \nonumber
& & (S(t,x,y,u)-I(t,x,y,u)+xyut)S(t,y,u,v);\end{eqnarray}

\vspace{-1cm}

\begin{eqnarray}\label{gf-irr-stat-eqn-but-two} I(t,x,y,u,v) &=&  xyuvt + \notag \\ 
&& (S(t,x,y,u)-I(t,x,y,u)+xyut)S(t,y,u,v).
\end{eqnarray}

Subtracting from (\ref{gf-all-stat-eqn-but-two}) the equation (\ref{gf-irr-stat-eqn-but-two}),  we obtain 
$$S(t,x,y,u,v)=I(t,x,y,u,v)+S(t,x,u,v)I(t,x,y,v)$$
and the statements in Theorem~\ref{dist-general-xyuv-thm} follow after observing that $S(t,x,y,u)-I(t,x,y,u)$ in \eqref{gf-irr-stat-eqn-but-two} records the joint distribution of the statistics $\lmax$, $\rmax$ and $\lmin$ on reducible separable permutations and hence is given by \eqref{reducible-three-stat}.

The initial terms of $S(t,x,y,u,v)$ are
\begin{footnotesize}
\begin{eqnarray}
&& t ( uvxy ) +
t^2 ( uv^2x^2y + u^2vxy^2 ) + \notag \\
&& t^3 ( uv^3x^3y + u^3vxy^3 + u^2v^2x^2y + u^2v^2xy^2 + u^2vx^2y^2 + uv^2x^2y^2 ) + \notag \\
&& t^4 ( uv^4x^4y + u^4vxy^4 + u^2v^3x^3y + uv^3x^3y^2 + u^3v^2xy^3 + u^3vx^2y^3 + u^3v^2x^2y + \notag \\ 
&& u^2v^3x^2y + u^2v^2x^3y + uv^3x^3y + u^3v^2xy^2 + u^2v^3xy^2 + u^3vx^2y^2 + 2u^2v^2x^2y^2 + \notag \\  
&& uv^3x^2y^2 + u^2vx^3y^2 + uv^2x^3y^2 + u^3vxy^3 + u^2v^2xy^3 + u^2vx^2y^3 + uv^2x^2y^3 ) + \cdots. \notag
\end{eqnarray}
\end{footnotesize}
The initial terms of $I(t,x,y,u,v)$ are
\begin{footnotesize}
\begin{eqnarray}
&& t ( uvxy ) +
t^2 (u^2vxy^2) + t^3 ( u^3vxy^3 + u^2v^2xy^2 + u^2vx^2y^2 ) + \notag \\
&& t^4 ( u^4vxy^4 + u^3v^2xy^3 + u^3vx^2y^3 + u^3v^2xy^2 + u^2v^3xy^2 + u^3vx^2y^2 + u^2v^2x^2y^2 + \notag \\
&& u^2vx^3y^2 + u^3vxy^3 + u^2v^2xy^3 + u^2vx^2y^3 ) +\cdots. \notag
\end{eqnarray}
\end{footnotesize}

\section{Unimodality conjectures}\label{final-sec}

\begin{table}
\begin{center}
\begin{tabular}{|c|c|c|c|c|c|c|c|c|}
\hline
\diagbox{$n$}{$k$}&  1  & 2  & 3 &  4 & 5 & 6 & 7 & 8\\
\hline
1&  1 & & & & & & & \\
\hline
2&  0 & 1 & & & & & & \\
\hline
3& 0 & 2 & 1 & & & & & \\
\hline
4 & 0  & 5 & 5 & 1 & & & & \\
\hline
5&  0 & 16  & 19 & 9 & 1 & & & \\
\hline
6& 0  & 60 & 73 & 49 & 14 & 1 & & \\
\hline
7&  0 & 248 & 298  & 232 & 104 & 20 & 1 & \\
\hline
8 & 0  & 1092 & 1288 & 1069 & 607 & 195 & 27 & 1 \\
\hline
\end{tabular}
  \caption{Distribution of $\rmax$ or $\lmin$ (resp.,  $\lmax$ or $\rmin$) on {\em irreducible} (resp., {\em reducible}) separable permutations, where $n$ is the length of permutations and  $k$ is the number of occurrences of the statistic.}
\label{tab-rmax-irr-sep}
\end{center}
\end{table}

We conclude our paper with unimodality conjectures based on the computational results in Tables~\ref{tab-rmax-sep},~\ref{tab-rmax-irr-sep} and~\ref{tab-rmin-irr-sep}. That would be interesting to find combinatorial proofs of the conjectures based on the appropriate embeddings, for example, of separable permutations with $k+1$ right-to-left maxima into separable permutations of the same length with $k$ right-to-left maxima for $k\geq 2$ and separable permutations ending with the largest element into separable permutations of the same length with two right-to-left maxima.

\begin{conj}\label{unimod-conj-rmax-sep} The distribution of separable permutations $\pi\in S_n$ with $\stat(\pi)=k$, $\stat\in\{\lmax,\rmax, \lmin, \rmin\}$, is unimodal with the peak being at $k=2$ for $n\geq 3$ (see Table~\ref{tab-rmax-sep}).\end{conj}

\begin{conj}\label{conj-2} The distribution of $\rmax$ or $\lmin$ (resp.,  $\lmax$ or $\rmin$) on {\em irreducible} (resp., {\em reducible}) separable permutations is unimodal with the peak being at $k=3$ for $n\geq 5$ (see Table~\ref{tab-rmax-irr-sep}).\end{conj}

\begin{conj}\label{conj-3} The distribution of $\rmax$ or $\lmin$ (resp., $\lmax$ or $\rmin$) on {\em reducible} (resp., {\em irreducible}) separable permutations is unimodal with the peak being at $k=1$ for $n\geq 1$ (see Table~\ref{tab-rmin-irr-sep}).\end{conj}

\begin{table}
\begin{center}
\begin{tabular}{|c|c|c|c|c|c|c|c|}
\hline
\diagbox{$n$}{$k$}&  1  & 2  & 3 &  4 & 5 & 6 & 7 \\
\hline
1&  1 & & & & & &  \\
\hline
2&  1 &  & & & & &  \\
\hline
3& 2 & 1 & & & & &  \\
\hline
4 & 6 & 4 & 1 &  & & &  \\
\hline
5&  22 & 15  & 7 & 1 &  & &  \\
\hline
6& 90 & 60 & 35 & 11 & 1 &  &  \\
\hline
7&  394 & 256  & 163  &  73 & 16  & 1 &   \\
\hline
8 & 1806 & 1148 & 758 & 406 & 138  &  22 & 1  \\
\hline
\end{tabular}
  \caption{Distribution of $\rmax$ or $\lmin$ (resp., $\lmax$ or $\rmin$) on {\em reducible} (resp., {\em irreducible}) separable permutations, where $n$ is the length of permutations and  $k$ is the number of occurrences of the statistic.}
\label{tab-rmin-irr-sep}
\end{center}
\end{table}

\vskip 4mm
\noindent {\bf Acknowledgments.}
The second author acknowledges the SUSTech International Center for Mathematics and extends special thanks to its director, Professor Efim Zelmanov, for their hospitality during the author's visit to the Center in November 2023. The work of the third author was supported by the National Science Foundation of China (No.\ 12171362).


\begin{thebibliography}{10}

\bibitem{AAV2011} M. H. Albert, M. D. Atkinson, V. Vatter. Subclasses of the separable permutations, {\em Bulletin of the London Math. Society} {\bf 43} (5) (2011) 859--870.

\bibitem{AHP2015} M. Albert, C. Homberger, J. Pantone. Equipopularity classes in the separable permutations. {\em Electron. J. Combin.} {\bf 22(2)} (2015), \#P2.2.

\bibitem{AJ2016} M. Albert and V. Jel\'{\i}nek. Unsplittable classes of separable permutations. {\em Electron. J. Combin.} {\bf 23(2)} (2016), \#P2.49.

\bibitem{BBL98} P. Bose, J. F. Buss, and A. Lubiw. Pattern matching for permutations. {\em Inform. Process. Lett.} {\bf 65} (1998) 5, 277--283.

\bibitem{BBS2010} M. Barnabei, F. Bonetti, M. Silimbani. The descent statistic on 123-avoiding permutations. {\em S\'em. Loth. de Comb.} {\bf 63} (2010), Article B63a.

\bibitem{BKLPRW} M. Bukata, R. Kulwicki, N. Lewandowski, L. Pudwell, J. Roth, T. Wheeland. Distributions of statistics over pattern-avoiding permutations. {\em J. Integer Sequences} {\bf 22} (2019), Article 19.2.6.

\bibitem{FuLinZeng} S. Fu, Z. Lin, and J. Zeng. On two new unimodal descent polynomials. {\em Discr. Math.} {\bf 341} (2018) 2616--2626.

\bibitem{GaoLiu} Y. Gao and S. Liu. Tree Enumeration Polynomials on Separable Permutations, {\em Enumerative Combin. and Appl.} {\bf 4:1} (2024) Article \#S2R6.

\bibitem{Kitaev2011Patterns} S. Kitaev. Patterns in permutations and words, Monographs in Theoretical Computer Science. An EATCS Series. Springer,  2011.


\bibitem{NRV} B.E. Neou, R. Rizzi, S. Vialette. Pattern Matching for Separable Permutations. In: Inenaga, S., Sadakane, K., Sakai, T. (eds) String Processing and Information Retrieval. SPIRE 2016. {\em Lecture Notes in Comp. Sci.} {\bf 9954} (2016).

\bibitem{Reif2003} A. Reifegerste. On the diagram of Schr\"{o}der permutations. {\em Electron. J. Combin.} {\bf 9(2)} (2002/03)Research paper 8, 23 pp.. Permutation patterns (Otago, 2003).

\bibitem{oeis} N.~J.~A.~Sloane. The Online Encyclopedia of Integer Sequences. Published electronically at \url{http://oeis.org}.

\bibitem{Stankova1994} Z. E. Stankova. Forbidden subsequences. {\em Discrete Math.} {\bf 132} (1994) 1--3, 291--316, 1994.
    
\bibitem{Stanley1989} R. P. Stanley. Enumerative combinatorics. Vol. 2, volume 62 of Cambridge Studies in Advanced Mathematics. Cambridge University Press, Cambridge, 1999.

\bibitem{West1995} J. West. Generating trees and the Catalan and Schr\"{o}der numbers. {\em Discrete Math.} {\bf 146} (1995) 1--3, 247--262.
     
\end{thebibliography}
\end{document}